\documentclass[a4paper]{amsproc}
\usepackage{amssymb}
\usepackage{mathrsfs}
\usepackage{amsfonts}
\usepackage{amsthm}
\usepackage{amsmath,amscd}
\usepackage{amsxtra}     
\usepackage{bm}
\usepackage[all,cmtip]{xy}
\usepackage{epsfig}
\usepackage{verbatim}
\usepackage{color}
\usepackage{enumerate}
\usepackage{hyperref}
\usepackage{tikz}
\usetikzlibrary{arrows,matrix,shapes,trees}
\usepackage{extarrows}

\theoremstyle{plain}
 \newtheorem{thm}{Theorem}[section]
 \newtheorem{prop}[thm]{Proposition}
 \newtheorem{lem}[thm]{Lemma}
 \newtheorem{cor}[thm]{Corollary}

 \newtheorem{lem'}[thm]{``Lemma''}
 
\theoremstyle{definition}

\theoremstyle{remark}
 \newtheorem{rmk}{Remark}[section]
 \numberwithin{equation}{section}

\newcommand{\N}{{\mathbb N}}
\newcommand{\Q}{{\mathbb Q}}
\newcommand{\Z}{{\mathbb Z}}

\newcommand{\F}{{\mathbb F}}
\newcommand{\Gm}{\mathbb{G}_{\mr{m}}}

\newcommand{\Gml}{\mathbb{G}_{\mr{m,log}}}

\newcommand{\mr}{\mathrm}
\newcommand{\mc}{\mathcal}

\newcommand{\Spec}{\mathop{\mr{Spec}}}

\makeatletter
\@namedef{subjclassname@2020}{\textup{2020} Mathematics Subject Classification}
\makeatother

\title[Comparison of Kummer topologies with classical topologies II]{Comparison of Kummer logarithmic topologies with classical topologies II}

\subjclass[2020]{14F20 (primary), 14A21 (secondary)}

\keywords{log schemes, Kummer log flat topology, comparison of cohomology}

\author[Heer Zhao]{\bfseries Heer Zhao}
\address{Fakult\"at f\"ur Mathematik, 
    Universit\"at Duisburg-Essen, 
    Essen 45117, 
    Germany}
\email{heer.zhao@uni-due.de}

\begin{document}

\vspace{18mm} \setcounter{page}{1} \thispagestyle{empty}

\begin{abstract}
We show that the higher direct images of smooth commutative group schemes from the Kummer log flat site to the classical flat site are torsion. For (1) smooth affine commutative schemes with geometrically connected fibers, (2) finite flat group schemes, (3) extensions of abelian schemes by tori, we give explicit description of the second higher direct image. If the rank of the log structure at any geometric point of the base is at most one, we show that the second higher direct image is zero for group schemes in case (1), case (3), and certain subcase of case (2). If the underlying scheme of the base is over $\Q$ or of characteristic $p>0$, we can also give more explicit description of the second higher direct image of group schemes in case (1), case (3), and certain subcase of case (3). Over standard Henselian log traits with finite residue field, we compute the first and the second Kummer log flat cohomology group with coefficients in group schemes in case (1), case (3), and certain subcase of case (3).
\end{abstract}

\maketitle

\section*{Introduction}
This article is a continuation of \cite{zha5}. For notation, conventions, and some history about Kummer logarithmic cohomology, we refer to that article. 

Let $X=(X,M_X)$ be an fs (fine and saturated) log scheme whose underlying scheme is locally noetherian, and let $G$ be a commutative group scheme over the underlying scheme of $X$. We endow $G$ with the induced log structure from $X$. Consider the ``forgetful'' map
$$\varepsilon_{\mr{fl}}:(\mr{fs}/X)_{\mr{kfl}}\rightarrow (\mr{fs}/X)_{\mr{fl}}$$
between these two sites. The first result of this article is the following theorem.

\begin{thm}[See also Theorem \ref{1.1}]
Assume that $G$ is smooth over $X$. Then the higher direct image $R^i\varepsilon_{\mr{fl}*}G$ is torsion for $i>0$.
\end{thm}

For the second higher direct image $R^2\varepsilon_{\mr{fl}*}G$, we give explicit description in the following three cases of the group scheme $G$.

\begin{thm}[See also Theorem \ref{2.6}]\label{0.2}
Assume that $G$ is smooth and affine over the underlying scheme of $X$, and has geometrically connected fibers over the underlying scheme of $X$. 
\begin{enumerate}[(1)]
\item We have 
$$R^2\varepsilon_{\mr{fl}*}G=\varinjlim_{n} (R^2\varepsilon_{\mr{fl}*}G)[n]=\bigoplus_{\text{$l$ prime}}(R^2\varepsilon_{\mr{fl}*}G)[l^\infty],$$
where $(R^2\varepsilon_{\mr{fl}*}G)[n]$ denotes the $n$-torsion subsheaf of $R^2\varepsilon_{\mr{fl}*}G$ and $(R^2\varepsilon_{\mr{fl}*}G)[l^\infty]$ denotes the $l$-primary part of $R^2\varepsilon_{\mr{fl}*}G$ for a prime number $l$.
\item The $l$-primary part $(R^2\varepsilon_{\mr{fl}*}G)[l^\infty]$ is supported on the locus where $l$ is invertible.
\item If $n$ is invertible on $X$, then 
$$(R^2\varepsilon_{\mr{fl}*}G)[n]=R^2\varepsilon_{\mr{fl}*}G[n]=G[n](-2)\otimes_{\Z}\bigwedge^2(\Gml/\Gm)_{X_{\mr{fl}}}.$$
\end{enumerate}
\end{thm}

\begin{thm}[See also Theorem \ref{2.7}]\label{0.3}
Assume that $G$ is finite and flat over the underlying scheme of $X$. Then we have
\begin{enumerate}[(1)]
\item $R^2\varepsilon_{\mr{fl}*}G=\bigoplus_{\text{$l$ prime}}(R^2\varepsilon_{\mr{fl}*}G)[l^\infty]$, where $(R^2\varepsilon_{\mr{fl}*}G)[l^\infty]$ denotes the $l$-primary part of $R^2\varepsilon_{\mr{fl}*}G$.
\item The $l$-primary part $(R^2\varepsilon_{\mr{fl}*}G)[l^\infty]$ is supported on the locus where $l$ is invertible.
\item Assume $X$ is connected, then the order of $G$ is constant on $X$. We denote the order by $n$ and assume that $n$ is invertible on $X$. Then 
$$R^2\varepsilon_{\mr{fl}*}G=G(-2)\otimes_{\Z}\bigwedge^2(\Gml/\Gm)_{X_{\mr{fl}}},$$
where $G(-2)$ denotes the tensor product of $G$ with $\Z/n\Z(-2)$.
\end{enumerate}
\end{thm}

\begin{thm}[See also Theorem \ref{2.9}]\label{0.4}
Assume that $G$ is an extension of an abelian scheme $A$ by a torus $T$ over $X$. Then we have
\begin{enumerate}[(1)]
\item $R^2\varepsilon_{\mr{fl}*}G=\bigoplus_{\text{$l$ prime}}(R^2\varepsilon_{\mr{fl}*}G)[l^\infty]$, where $(R^2\varepsilon_{\mr{fl}*}G)[l^\infty]$ denotes the $l$-primary part of $R^2\varepsilon_{\mr{fl}*}G$.
\item We have $(R^2\varepsilon_{\mr{fl}*}G)[n]=R^2\varepsilon_{\mr{fl}*}G[n]$.
\item The $l$-primary part $(R^2\varepsilon_{\mr{fl}*}G)[l^\infty]$ is supported on the locus where $l$ is invertible.
\item If $n$ is invertible on $X$, then 
$$(R^2\varepsilon_{\mr{fl}*}G)[n]=R^2\varepsilon_{\mr{fl}*}G[n]=G[n](-2)\otimes_{\Z}\bigwedge^2(\Gml/\Gm)_{X_{\mr{fl}}},$$
where $G[n](-2)$ denotes the tensor product of $G[n]$ with $\Z/n\Z(-2)$.
\end{enumerate}
\end{thm}

As applications of Theorem \ref{0.2}, Theorem \ref{0.3}, and Theorem \ref{0.4}, we have more explicit description of $R^2\varepsilon_{\mr{fl}*}G$ in certain special cases of the base.

The first special case concerns the ranks of the log structure at geometric points of the base. If $Y$ is an fs log scheme over $X$ such that the rank $(M_Y^{\mr{gp}}/\mc{O}_Y^\times)_y$ is at most one for any geometric point $y$ of $Y$, then the restriction of the sheaf $\bigwedge^2(\Gml/\Gm)_{X_{\mr{fl}}}$ to $(\mr{st}/Y)$ is zero, where $(\mr{st}/Y)$ denotes the full subcategory of $(\mr{fs}/X)$ consisting of strict fs log schemes over $Y$. Thus Theorem \ref{0.2} (especially part (3)), Theorem \ref{0.3} (especially part (3)), and Theorem \ref{0.4} (especially part (4)) imply the following theorem.

\begin{thm}[See also Theorem \ref{2.10}]\label{0.5}
Assume that $G$ satisfies any of the following three conditions:
\begin{enumerate}[(i)]
\item $G$ is smooth and affine over the underlying scheme of $X$, and has geometrically connected fibers over the underlying scheme of $X$;
\item $G$ is finite flat of order $n$ over the underlying scheme of $X$, and for any prime number $l$ the kernel of $G\xrightarrow{l^n}G$ is also finite flat over the underlying scheme of $X$;
\item $G$ is an extension of an abelian scheme by a torus over the underlying scheme of $X$.
\end{enumerate}
Let $Y\in (\mr{fs}/X)$ be such that the ranks of the stalks of the \'etale sheaf $M_Y^{\mr{gp}}/\mc{O}_Y^{\times}$ are at most one, and let $(\mr{st}/Y)$ be the full subcategory of $(\mr{fs}/X)$ consisting of strict fs log schemes over $Y$. Then the restriction of $R^2\varepsilon_{\mr{fl}*}G$ to $(\mr{st}/Y)$ is zero.
\end{thm}

The rest two special cases involve the characteristic of the base.

Assume that the underlying scheme of $X$ is over $\Q$, i.e. any positive integer $n$ is invertible on $X$. Then we have an explicit description of $(R^2\varepsilon_{\mr{fl}*}G)[n]$ for any $n$ in the case (i) and the case (iii) from Theorem \ref{0.5} by Theorem \ref{0.2} (3) and Theorem \ref{0.4} (4) respectively, as well as an explicit description of $R^2\varepsilon_{\mr{fl}*}G$ in the case (ii) from Theorem \ref{0.5} by Theorem \ref{0.3} (3). Thus we get the following theorem.

\begin{thm}[See also Theorem \ref{2.11}]
Assume that the underlying scheme of $X$ is a $\Q$-scheme, and let $G$ be as in any of the three cases of Theorem \ref{0.5}.
Then we have 
\[R^2\varepsilon_{\mr{fl}*}G=\varinjlim_{n}G[n](-2)\otimes_{\Z}\bigwedge^2(\Gml/\Gm)_{X_{\mr{fl}}}\]
in cases (i) and (iii), and
\[R^2\varepsilon_{\mr{fl}*}G=G(-2)\otimes_{\Z}\bigwedge^2(\Gml/\Gm)_{X_{\mr{fl}}}\]
in case (ii).
\end{thm}

Assume that the underlying scheme of $X$ is an $\F_p$-scheme for a prime number $p$. Then any prime $l\neq p$ is invertible on $X$, hence we have an explicit description of $(R^2\varepsilon_{\mr{fl}*}G)[l^\infty]$ in the case (i) and the case (iii) from Theorem \ref{0.5} by Theorem \ref{0.2} (3) and Theorem \ref{0.4} (4) respectively. And we also have $(R^2\varepsilon_{\mr{fl}*}G)[p^\infty]=0$ in these two cases by Theorem \ref{0.2} (2) and Theorem \ref{0.4} (3) respectively. In a suitable modified version of the case (ii) from Theorem \ref{0.5}, we have similar results. Therefore we get the following theorem.

\begin{thm}[See also Theorem \ref{2.12}]
Let $p$ be a prime number. Assume that the underlying scheme of $X$ is an $\F_p$-scheme, and $G$ satisfies one of the following three conditions:
\begin{enumerate}[(i)]
\item $G$ is smooth and affine over the underlying scheme of $X$, and has geometrically connected fibers over the underlying scheme of $X$;
\item $G$ is finite flat of order $n$ over the underlying scheme of $X$, and the kernel $G[p^n]$ of $G\xrightarrow{p^n}G$ is also finite flat over the underlying scheme of $X$;
\item $G$ is an extension of an abelian scheme by a torus over the underlying scheme of $X$.
\end{enumerate}
Then we have 
$$R^2\varepsilon_{\mr{fl}*}G=\varinjlim_{(r,p)=1}G[r](-2)\otimes_{\Z}\bigwedge^2(\Gml/\Gm)_{X_{\mr{fl}}}$$
in cases (i) and (iii), and 
$$R^2\varepsilon_{\mr{fl}*}G=(G/G[p^n])(-2)\otimes_{\Z}\bigwedge^2(\Gml/\Gm)_{X_{\mr{fl}}}$$
in case (ii).
\end{thm}

At last we make some computations in Section 3. In particular, over standard Henselian log traits with finite residue field, we compute the first and the second Kummer log flat cohomology with coefficients in group schemes which satisfy any of the conditions of Theorem \ref{0.5}.

\section{The higher direct images are torsion}\label{sec1}
Let $X$ be an fs log scheme whose underlying scheme is locally noetherian. Let $G$ be a smooth commutative group scheme over the underlying scheme of $X$. In this section, we investigate the higher direct images $R^i\varepsilon_{\mr{fl}*}G$ of $G$. The main result is the following theorem.

\begin{thm}\label{1.1}
Let $X$ be an fs log scheme whose underlying scheme is locally noetherian. Let $G$ be a smooth commutative group scheme over the underlying scheme of $X$. Then the higher direct image $R^i\varepsilon_{\mr{fl}*}G$ is a torsion sheaf for any $i>0$.
\end{thm}

Apparently the torsionness of $R^i\varepsilon_{\mr{fl}*}G$ is reduced to the torsionness of $H^i_{\mr{kfl}}(X,G)$ in the case that the underlying scheme of $X$ is $\Spec R$ with $R$ a strictly henselian local ring. We are going to investigate the torsionness of $H^i_{\mr{kfl}}(X,G)$ in this case using \v{C}ech cohomology.

Now let the underlying scheme of $X$ be $\Spec R$ with $R$ a strictly henselian local ring. Let $x$ denote the closed point, let $k$ be the residue field of $R$, and let $P\to M_X$ be a chart of the log structure of $X$ with $P$ an fs monoid, such that $P\xrightarrow{\cong}M_{X,x}/\mc{O}_{X,x}^{\times}$.

Let $P^{1/n}$ denote the monoid $P$ regarded as a monoid above $P$ via the homomorphism $P\xrightarrow{n} P$. Let $X_n:=X\times_{\Spec\Z[P]}\Spec\Z[P^{1/n}]$ endowed with the canonical log structure associated to $P^{1/n}$ and let $H_n$ denote the group scheme $\Spec \Z[(P^{1/n})^{\mr{gp}}/P^{\mr{gp}}]$ over $\Spec\Z$, then $X_n$ is a Kummer log flat cover of $X$ such that $X_n\times_XX_n\cong X_n\times_{\Spec\Z}H_n$. We regard $x$ as a log point with respect to the induced log structure. Let $x_n:=X_n\times_Xx$. Then $x_n$ is a Kummer log flat cover of $x$ such that $x_n\times_xx_n\cong x_n\times_{\Spec\Z}H_n$. 

The following lemma, which is taken from \cite{zha5}, identifies the \v{C}ech cohomology group $\check{H}_{\mr{kfl}}^i(X_n/X,G)$ with $\check{H}_{\mr{kfl}}^i(x_n/x,G)$ for any $i>0$.

\begin{lem}\label{1.2}
Let $X$ be an fs log scheme such that its underlying scheme is $\Spec R$ with $R$ a strictly henselian noetherian local ring. Let $k$ be the residue field of $R$, $x$ the closed point of $X$, and $P_X\to M_X$ a chart of the log structure $M_X$ of $X$ with $P$ an fs monoid (here $P_X$ denotes the constant sheaf associated to $P$), such that $P\xrightarrow{\cong}M_{X,x}/\mc{O}_{X,x}^{\times}$. Let $P^{1/n}$, $X_n$, and $x_n$ be as defined above. Let $G$ be a smooth commutative group scheme over the underlying scheme of $X$, and we endow $G$ with the induced log structure from $X$. Then the canonical map 
$$\check{H}_{\mr{kfl}}^i(X_n/X,G)\rightarrow \check{H}_{\mr{kfl}}^i(x_n/x,G)$$
is an isomorphism for any $i>0$.
\end{lem}
\begin{proof}
See \cite[Lem. 3.12]{zha5}.
\end{proof}

Before going to the next lemma, we recall the standard complex  for computing the Hochschild cohomology from \cite[Expos\'e I, \S 5.1]{sga3-1}. Let $H$ be a group scheme over a base scheme $S$, and let $F$ be a functor of abelian groups on the flat site of $S$ endowed with an $H$-action. The standard complex 
\begin{equation}\label{eq1.1}
C^{\bullet}_S(H,F):C^{0}_S(H,F)\xrightarrow{d^0}C^{1}_S(H,F)\xrightarrow{d^1}C^{2}_S(H,F)\xrightarrow{d^2}\cdots
\end{equation}
for computing the Hochschild cohomology $H^{\bullet}_{S_{\mr{fl}}}(H,F)$ is given by:
$$C^{i}_S(H,F):=\mr{Mor}_S(H^i,F),$$
and the differential 
$$d^i:C^{i}_S(H,F)\to C^{i+1}_S(H,F)$$
is given by the alternating sum $\partial_0-\partial_1+\cdots+(-1)^{i+1}\partial_{i+1}$, where 
$$\partial_j:C^{i}_S(H,F)\to C^{i+1}_S(H,F)$$
sends $c\in C^{i}_S(H,F)$ to $\partial_j(c)$ defined by
$$\partial_j(c)(h_1,\cdots,h_{i+1})=\begin{cases}h_1\cdot c(h_2,\cdots,h_{i+1}),&\text{if $j=0$;}\\
c(h_1,\cdots,h_j\cdot h_{j+1},\cdots,h_{i+1}),&\text{if $1\leq j\leq i$;}\\
c(h_1,\cdots,h_{i}),&\text{if $j=i+1$.}\end{cases}$$

The following lemma, which is implicitly proved in \cite[\S 4.11, \S 4.12]{kat2}, can be used to describe the \v{C}ech cohomology group $\check{H}_{\mr{kfl}}^i(x_n/x,G)$ in terms of the Hochschild cohomology group $H^i_{x_{\mr{fl}}}(H_n,G)$ for $i\geq1$, where the action of $H_n$ on $G$ is the trivial one.

\begin{lem}\label{1.3}
Let $X$, $R$, $P$, $P^{1/n}$, $X_n$, and $G$ be as in Lemma \ref{1.2}. We further assume that $R$ is artinian. Let $F$ be the sheaf of abelian groups on $(\mr{fs}/X)_{\mr{kfl}}$ defined by $T\mapsto \Gamma(T\times_XX_n,G)$. Then we have the following.
\begin{enumerate}
\item  The group scheme $H_n=\Spec\Z[(P^{1/n})^{\mr{gp}}/P^{\mr{gp}}]$ acts on $F$.
\item For any $i\geq0$, the \v{C}ech cohomology group $\check{H}_{\mr{kfl}}^i(X_n/X,G)$ is canonically identified with the Hochschild cohomology $H^i_{X_{\mr{fl}}}(H_n,F)$ of $H_n$ with coefficients in $F$ on the flat site $X_{\mr{fl}}=(\mr{fs}/X)_{\mr{fl}}$ (see \cite[Expos\'e I, \S 5.1]{sga3-1} for the definition of Hochschild cohomology).
\item We endow $G$ with the trivial $H_n$-action. Then the canonical homomorphism $G\to F$ induced by the projection $T\times_XX_n\to T$ is $H_n$-equivariant.
\item The homomorphism $G\to F$ induces an isomorphism 
$$H^i_{X_{\mr{fl}}}(H_n,G)\xrightarrow{\cong}H^i_{X_{\mr{fl}}}(H_n,F)$$
for each $i\geq1$.
\end{enumerate}
\end{lem}
\begin{proof}
See \cite[\S 4.11, \S 4.12]{kat2}.
\end{proof}

\begin{cor}\label{1.4}
Let $X$, $X_n$, $x$, and $G$ be as in Lemma \ref{1.2}. Then the \v{C}ech cohomology group $\check{H}_{\mr{kfl}}^i(X_n/X,G)$ is canonically identified with the Hochschild cohomology group $H^i_{x_{\mr{fl}}}(H_n,G)$ for any $i\geq1$, where $G$ is regarded as an $H_n$-module with respect to the trivial action.
\end{cor}
\begin{proof}
By Lemma \ref{1.2}, we have a canonical isomorphism 
$$\check{H}_{\mr{kfl}}^i(X_n/X,G)\xrightarrow{\cong} \check{H}_{\mr{kfl}}^i(x_n/x,G).$$
Since the log point $x$ clearly satisfies the extra condition of Lemma \ref{1.3}, we get a canonical isomorphism
$$\check{H}_{\mr{kfl}}^i(x_n/x,G)\xrightarrow{\cong}H^i_{x_{\mr{fl}}}(H_n,G).$$
Then the result follows.
\end{proof}
 
Now we turn to the study of $H^i_{x_{\mr{fl}}}(H_n,G)$.

\begin{lem}\label{1.5}
Let $k$, $G$, and $x$ be as in Lemma \ref{1.2}. Let $p$ be the characteristic of the field $k$, and let $n=m\cdot p^r$ with $(m,p)=1$. By the definition of Hochschild cohomology, see \cite[Expos\'e I, \S 5.1]{sga3-1}, the inclusion $H_{p^r}\to H_{n}$ induces a canonical map
$$\mr{Res}:H^i_{x_{\mr{fl}}}(H_n,G)\to H^i_{x_{\mr{fl}}}(H_{p^r},G),$$
which we call the restriction map. Then the multiplication by $m$ map on $H^i_{x_{\mr{fl}}}(H_n,G)$ factors through the restriction map $\mr{Res}$.
\end{lem}
\begin{proof}
Note that the analogous result for the abstract group cohomology is well-known. We are going to show that the standard proof in the setting of the abstract group cohomology also works in this special case of Hochschild cohomology here.

Firstly we construct a corestriction map 
$$\mr{Cor}:H^i_{x_{\mr{fl}}}(H_{p^r},G)\to H^i_{x_{\mr{fl}}}(H_{n},G)$$
for each $i\geq0$. Since the field $k$ is separable closed and $(p,m)=1$, we have $H_n\cong H_m\times H_{p^r}$ with $H_m$ a constant group scheme. Let $\iota$ denote the inclusion $H_{p^r}\hookrightarrow H_n$. We claim that $\{H^i_{x_{\mr{fl}}}(H_{p^r},\iota^*(-))\}_i$ is a universal $\delta$-functor on the category of $H_n$-modules. It suffices to show that for any $H_n$-module $M$, we have $H^i_{x_{\mr{fl}}}(H_{p^r},\iota^*(E_{H_n}(M)))=0$, where $E_{H_n}(M)$ is the $H_n$-module associated to $M$ defined in \cite[Expos\'e I, Def. 5.2.0]{sga3-1} which contains $M$ as an $H_n$-submodule. But this follows from \cite[Expos\'e I, Lem. 5.2.2]{sga3-1} by the reexpression 
$$\iota^*(E_{H_n}(M))=E_{H_{p^r}}(\underline{\mr{Mor}}_k(H_m,M)).$$
Now for any $H_n$-module $M$, we define the corestriction map in degree 0 to be the norm map 
$$\mr{Norm}:M^{H_{p^r}}\to M^{H_n},a\mapsto\sum_{g\in H_m(k)}g\cdot a .$$
It extends uniquely to a morphism Cor of the universal $\delta$-functor $\{H^i_{x_{\mr{fl}}}(H_{p^r},\iota^*(-))\}_i$ into the universal $\delta$-functor $\{H^i_{x_{\mr{fl}}}(H_{n},-)\}_i$. Since the composition 
$$M^{H_n}\hookrightarrow M^{H_{p^r}}\xrightarrow{\mr{Norm}} M^{H_n}$$
is the multiplication by $m$ map, so is the composition 
$$H^i_{x_{\mr{fl}}}(H_n,M)\xrightarrow{\mr{Res}} H^i_{x_{\mr{fl}}}(H_{p^r},M)\xrightarrow{\mr{Cor}}H^i_{x_{\mr{fl}}}(H_n,M).$$
Taking $M=G$ finishes the proof.
\end{proof}

Next we prove that $H^i_{x_{\mr{fl}}}(H_{p^r},G)$ is torsion for $i>0$, which would imply that $H^i_{x_{\mr{fl}}}(H_{n},G)$ is torsion for $i>0$ by Lemma \ref{1.5}.

\begin{lem}\label{1.6}
Let $A$ be a local artinian $k$-algebra with residue field $k$ such that its maximal ideal $\mathfrak{m}_A$ satisfies $\mathfrak{m}_A^t=0$ for some positive integer $t$. Let $\pi:G(A)\to G(k)$ (resp. $\tau:G(k)\to G(A)$) be the map induced by the projection $A\to k$ (resp. the inclusion $k\to A$). Then the multiplication by $p^t$ map on $G(A)$ agrees with $\tau\circ[p^t]\circ\pi$.
\end{lem}
\begin{proof}
Consider the following commutative diagram
$$\xymatrix{
G(A)\ar[r]^{\pi}\ar[d]_{p^t} &G(k)\ar[d]^{p^t}\ar@/^1pc/[l]^{\tau}  \\
G(A)\ar[r]^{\pi} &G(k)\ar@/^1pc/[l]^{\tau}
}.$$
For any $a\in G(A)$, we have $a-\tau(\pi(a))\in\mr{Ker}(\pi)$. Hence 
$$p^t\cdot a-\tau(p^t\cdot\pi(a))=p^t\cdot (a-\tau(\pi(a)))=0$$
by \cite[Lem. 1.1.1]{katz1}.
\end{proof}

\begin{lem}\label{1.7}
Let $k$, $P$, $G$, and $x$ be as in Lemma \ref{1.2}. Let $p$ be the characteristic of the field $k$. Assume that $P^{\mr{gp}}\cong\Z^a$. Then the group $H^i_{x_{\mr{fl}}}(H_{p^r},G)$ is killed by $p^{iap^r}$ for any $i>0$.
\end{lem}
\begin{proof}
Now we have no analogous corestriction map available for the inclusion $H_1\hookrightarrow H_{p^r}$. Instead we are going to show that the map 
$$p^{p^r}:C^{\bullet}_x(H_{p^r},G)\to C^{\bullet}_x(H_{p^r},G)$$
factors through the canonical restriction map $C^{\bullet}_x(H_{p^r},G)\to C^{\bullet}_x(H_{1},G)$, where $C^{\bullet}_x(H_{p^r},G)$ (resp. $C^{\bullet}_x(H_{1},G)$) is the standard complex computing the group cohomology of the trivial $H_n$-module (resp. $H_1$-module) $G$ over $x$.

The complex $C^{\bullet}_x(H_{1},G)$ for the trivial group $H_1$ is the complex 
$$G(k)\xrightarrow{0}G(k)\xrightarrow{\mr{id}}G(k)\xrightarrow{0}G(k)\xrightarrow{\mr{id}}G(k)\xrightarrow{0}\cdots,$$
so we have $H^i_{x_{\mr{fl}}}(H_1,G)=0$ for $i>0$. Consider the following diagram
\begin{equation}\label{eq1.2}
\xymatrix{
C^0_x(H_{p^r},G)\ar[r]\ar[d]^{\mr{Res}} &C^1_x(H_{p^r},G)\ar[r]\ar[d]^{\mr{Res}} &C^2_x(H_{p^r},G)\ar[r]\ar[d]^{\mr{Res}} &\cdots  \\
G(k)\ar[r]^0\ar[d]^{p^{iap^r}} &G(k)\ar@{=}[r]\ar[d]^{p^{iap^r}} &G(k)\ar[r]^0\ar[d]^{p^{iap^r}} &\cdots  \\
G(k)\ar[r]^0\ar[d]^{\tau} &G(k)\ar@{=}[r]\ar[d]^{\tau} &G(k)\ar[r]^0\ar[d]^{\tau} &\cdots  \\
C^0_x(H_{p^r},G)\ar[r] &C^1_x(H_{p^r},G)\ar[r] &C^2_x(H_{p^r},G)\ar[r] &\cdots
},
\end{equation}
where the upper vertical maps are induced by the inclusion $H_1\hookrightarrow H_{p^r}$ and the maps $\tau$ are as in Lemma \ref{1.6}.
The maximal ideal of the local artinian ring underlying $H_{p^r}^i$ is killed by raising to $iap^r$-th power. By Lemma \ref{1.6}, the composition of the maps on the column at degree $i$ is the multiplication by $p^{iap^r}$ map. The upper squares and the middle squares are clearly commutative, we claim that the lower squares are also commutative. We choose the following square
$$\xymatrixcolsep{7pc}\xymatrix{G(k)\ar[r]^{1_{G(k)}-1_{G(k)}+\cdots +(-1)^n1_{G(k)}}\ar[d]^{\tau} &G(k)\ar[d]^{\tau}  \\
\mr{Mor}_k(H_{p^r}^{n-1},G)\ar[r]^{\partial_0-\partial_1+\cdots+(-1)^n\partial_n} &\mr{Mor}_k(H_{p^r}^n,G)
}$$
to check. It suffices to check the commutativity of the square
$$\xymatrix{G(k)\ar[r]^{1_{G(k)}}\ar[d]^{\tau} &G(k)\ar[d]^{\tau}  \\
\mr{Mor}_k(H_{p^r}^{n-1},G)\ar[r]^{\partial_j} &\mr{Mor}_k(H_{p^r}^n,G)
}$$
for each $j$. For $a:\Spec k\to G$ and $0<j<n$, we have that $\partial_j(\tau(a))$ is given by the composition 
$$H_{p^r}^n\xrightarrow{\iota_j}H_{p^r}^{n-1}\xrightarrow{\pi_{n-1}}\Spec k\xrightarrow{a}G,$$
where $\iota_j$ is the map collapsing the $j$-th term and the $(j+1)$-th term via the group law of $H_{p^r}$ and $\pi_{n-1}$ is the structure map of $H_{p^r}^{n-1}$ as a $k$-scheme. Since $\pi_{n-1}\circ\iota_j$ is the structure map $\pi_n$ of $H_{p^r}^{n}$, the commutativity follows. The cases $j=0$ and $j=n$ can be shown similarly. Hence the diagram (\ref{eq1.2}) induces a factorization of $p^{iap^r}:H^i_{s_{\mr{fl}}}(H_{p^r},G)\to H^i_{s_{\mr{fl}}}(H_{p^r},G)$ through $H^i_{s_{\mr{fl}}}(H_{1},G)=0$. This finishes the proof.
\end{proof}

\begin{cor}\label{1.8}
Let $X$, $X_n$, $x$, and $G$ be as in Lemma \ref{1.2}. Then the \v{C}ech cohomology group $\check{H}_{\mr{kfl}}^i(X_n/X,G)$ is torsion for any $i\geq1$.
\end{cor}
\begin{proof}
By Corollary \ref{1.4}, it suffices to show that the Hochschild cohomology group $H^i_{x_{\mr{fl}}}(H_n,G)$ is torsion for any $i\geq1$. The latter follows from Lemma \ref{1.5} and Lemma \ref{1.7}.
\end{proof}

\begin{lem}\label{1.9}
Let $X$ and $X_n$ be as in Lemma \ref{1.2}. Let $F$ be a sheaf on $(\mr{fs}/X)_{\mr{kfl}}$. Then the family $\mathscr{X}_{\N}:=\{X_n\rightarrow X\}_{n\geq1}$ of Kummer log flat covers of $X$ satisfies the condition (L3) from \cite[\S 2]{art3}, whence a \v{C}ech-to-derived functor spectral sequence
\begin{equation}\label{eq1.3}
\check{H}_{\mr{kfl}}^i(\mathscr{X}_{\N},\underline{H}_{\mr{kfl}}^j(F))\Rightarrow H^{i+j}_{\mr{kfl}}(X,F),
\end{equation}
where 
$$\check{H}_{\mr{kfl}}^i(\mathscr{X}_{\N},-):=\varinjlim_{n\geq1}\check{H}_{\mr{kfl}}^i(X_n/X,-).$$
\end{lem}
\begin{proof}
This follows from \cite[Chap. II, Sec. 3, (3.3)]{art3}.
\end{proof}

Now we are ready to prove Theorem \ref{1.1}.

\begin{proof}[Proof of Theorem \ref{1.1}]
The case $i=1$ for a smooth affine commutative group scheme follows from Kato's computation of $R^1\varepsilon_{\mr{fl}*}G$, and the case $i=1$ for a general smooth commutative group scheme follows from \cite[Thm. 3.14]{zha5}. We proceed with induction on $i$. Assume that $R^i\varepsilon_{\mr{fl}*}G$ is torsion for $0<i<N+1$, we are going to show that $R^{N+1}\varepsilon_{\mr{fl}*}G$ is torsion. 

It suffices to show that the group $H^{N+1}_{\mr{kfl}}(X,G)$ is torsion in the case that the underlying scheme of $X$ is $\Spec R$ with $R$ a strictly henselian noetherian local ring. Let $x$ denote the closed point, let $k$ be the residue field of $R$, and let $P\to M_X$ be a chart of the log structure of $X$ with $P$ some fs monoid, such that $P\xrightarrow{\cong}M_{X,x}/\mc{O}_{X,x}^{\times}$. Let $p$ be the characteristic of $k$. We have a spectral sequence 
$$E_2^{i,j}=\check{H}^i_{\mr{kfl}}(\mathscr{X}_{\N},\underline{H}^j_{\mr{kfl}}(G))\Rightarrow H^{i+j}_{\mr{kfl}}(X,G)$$
by Lemma \ref{1.9}. Let $\gamma\in H^{N+1}_{\mr{kfl}}(X,G)$. By \cite[Prop. 3.1]{zha5}, we can find a Kummer log flat cover $T\rightarrow X$ such that $\gamma$ dies in $H_{\mr{kfl}}^{N+1}(T,G)$. By \cite[Cor. 2.16]{niz1}, we may assume that for some $m$, we have a factorization $T\rightarrow X_m\rightarrow X$ with $T\rightarrow X_m$ a classical flat cover. It follows that $\gamma$ regarded as a class on $X_m$ is trivialized by a classical flat cover, i.e. 
$$\gamma\in\mr{ker}(H_{\mr{kfl}}^{N+1}(X_m,G)\rightarrow H_{\mr{fl}}^0(X_m,R^{N+1}\varepsilon_{\mr{fl}*} G)).$$
Since $R^i\varepsilon_{\mr{fl}*}G$ is torsion for $0<i<N+1$ by induction hypothesis and 
$$H^{N+1}_{\mr{fl}}(X_m,G)=H^{N+1}_{\mr{\acute{e}t}}(X_m,G)=0,$$
the Leray spectral sequence 
$$E_2^{i,j}=H^i_{\mr{fl}}(X_m,R^j\varepsilon_{\mr{fl}*}G)\Rightarrow H^{i+j}_{\mr{kfl}}(X_m,G)$$
implies that $\gamma$ is killed by a positive integer $a$ in $H^{N+1}_{\mr{kfl}}(X_m,G)$. It follows that 
$$a\gamma\in \mr{Ker}(H^{N+1}_{\mr{kfl}}(X,G)\to \check{H}^0_{\mr{kfl}}(\mathscr{X}_{\N},\underline{H}^{N+1}_{\mr{kfl}}(G))).$$

Therefore to show that $\gamma$ is torsion, it suffices to show that the groups 
$$\check{H}^i_{\mr{kfl}}(\mathscr{X}_{\N},\underline{H}^{N+1-i}_{\mr{kfl}}(G))$$
are torsion for $0<i\leq N+1$. By Corollary \ref{1.8}, we are left with showing that the groups $\check{H}^i_{\mr{kfl}}(\mathscr{X}_{\N},\underline{H}^{N+1-i}_{\mr{kfl}}(G))$ are torsion for $0<i<N+1$. Further it suffices to show that the group  $\check{H}^i_{\mr{kfl}}(X_n/X,\underline{H}^{N+1-i}_{\mr{kfl}}(G))$ is torsion for any $0<i<N+1$ and any $n$.

The torsionness of $R^t\varepsilon_{\mr{fl}*}G$ for $0<t<N+1$, 
$$H^r_{\mr{fl}}(\underbrace{X_n\times_X\cdots\times_XX_n}_\text{$d+1$ folded},G)=H^r_{\mr{fl}}(X_n\times_{\Spec\Z} H_n^d,G)=H^r_{\mr{\acute{e}t}}(X_n\times_{\Spec\Z} H_n^d,G)=0$$
for any $r>0$ and any $d\geq0$, and the Leray spectral sequence 
$$E_2^{u,v}=H^u_{\mr{fl}}(\underbrace{X_n\times_X\cdots\times_XX_n}_\text{$d+1$ folded},R^v\varepsilon_{\mr{fl}*}G)\Rightarrow H^{v+v}_{\mr{kfl}}(\underbrace{X_n\times_X\cdots\times_XX_n}_\text{$d+1$ folded},G)$$
together imply that $H^{N+1-i}_{\mr{kfl}}(\underbrace{X_n\times_X\cdots\times_XX_n}_\text{$d+1$ folded},G)$ is torsion for $0<i<N+1$. It follows that $\check{H}^i_{\mr{kfl}}(X_n/X,\underline{H}^{N+1-i}_{\mr{kfl}}(G))$ is torsion for any $0<i<N+1$ and any $n$. This finishes the proof.
\end{proof}

\section{The second higher direct image}
We have shown that the higher direct images of a smooth commutative group scheme is torsion. In this section, we present further descriptions of the second higher direct image.

The following theorem is taken from \cite{zha5}.

\begin{thm}\label{2.1}
Let $X$ be an fs log scheme such that its underlying scheme is $\Spec R$ with $R$ a noetherian strictly Henselian local ring. Let $P_X\to M_X$ be a chart of the log structure $M_X$ of $X$ with $P$ an fs monoid (here $P_X$ denotes the constant sheaf associated to $P$), such that $P\xrightarrow{\cong}M_{X,x}/\mc{O}_{X,x}^{\times}$. Let $G$ be a smooth commutative group scheme over the underlying scheme of $X$, and we endow $G$ with the inverse image log structure. Let $X_n$ be as constructed in the beginning of Section \ref{sec1}, $x$ the closed point of $X$, and $x_n:=X_n\times_Xx$. Then the canonical map
$$\varinjlim_n\check{H}^2_{\mr{kfl}}(X_n/X,G)\xrightarrow{\cong}H^2_{\mr{kfl}}(X,G)$$
given by the spectral sequence (\ref{eq1.3}) is an isomorphism.
\end{thm}
\begin{proof}
See \cite[Thm. 3.19]{zha5}.
\end{proof}

\begin{cor}\label{2.2}
Let the notation and the assumptions be as in Theorem \ref{2.1}. Then the canonical map 
$$H^2_{\mr{kfl}}(X,G)\xrightarrow{\cong}H^2_{\mr{kfl}}(x,G)$$
is an isomorphism.
\end{cor}
\begin{proof}
In the canonical commutative diagram
$$\xymatrix{
\varinjlim_n\check{H}^2_{\mr{kfl}}(X_n/X,G)\ar[r]^\cong\ar[d]^\cong &\varinjlim_n\check{H}^2_{\mr{kfl}}(x_n/x,G)\ar[d]^\cong \\
H^2_{\mr{kfl}}(X,G)\ar[r] &H^2_{\mr{kfl}}(x,G)
}$$
the vertical maps are isomorphisms by Theorem \ref{2.1} and the upper horizontal map is an isomorphism by Lemma \ref{1.2}. Then the result follows.
\end{proof}

\subsection{The case of smooth affine commutative group schemes with connected fibers}
The following lemma is a generalization of \cite[Lem. 3.20]{zha5} from tori to smooth affine commutative group schemes.

\begin{lem}\label{2.3}
Let the notation and the assumptions be as in Theorem \ref{2.1}. Let $x$ be the closed point of $X$, and $p$ the characteristic of the residue field $k$ of $R$. We further assume that the closed fiber $G\times_Xx$ of $G$ is affine and connected. Then we have
\begin{enumerate}[(1)]
\item $\check{H}_{\mr{kfl}}^2(X_n/X,G)=\check{H}_{\mr{k\acute{e}t}}^2(X_n/X,G)$ for $(n,p)=1$;
\item $\check{H}_{\mr{kfl}}^2(X_{p^r}/X,G)=0$ for $r>0$;
\item $$H_{\mr{kfl}}^2(X,G)=\varinjlim_{(n,p)=1}\check{H}_{\mr{kfl}}^2(X_n/X,G)=\varinjlim_{(n,p)=1}\check{H}_{\mr{k\acute{e}t}}^2(X_n/X,G)=H_{\mr{k\acute{e}t}}^2(X,G),$$
in particular $H_{\mr{kfl}}^2(X,G)$ is torsion, $p$-torsion-free and divisible.
\end{enumerate}
\end{lem}
\begin{proof}
If $(n,p)=1$, $X_n\to X$ is a Kummer log \'etale cover. Then part (1) is clear by the definition of \v{C}ech cohomology.

We have $\check{H}_{\mr{kfl}}^2(X_{p^r}/X,G)=H^2_{x_{\mr{fl}}}(H_{p^r},G)$ for $r>0$ by Corollary \ref{1.4}, where $H_{p^r}$ is as in Corollary \ref{1.4}. We are reduced to show that $H^2_{x_{\mr{fl}}}(H_{p^r},G)=0$. Let 
$$\mr{Ext}_{\mr{alg}}(H_{p^r}\times_Xx,G\times_Xx)$$
denote the group of isomorphism extension classes of $H_{p^r}\times_Xx$ by $G\times_Xx$, and let
$$\mr{Ext}_{\mr{s}}(H_{p^r}\times_Xx,G\times_Xx)$$ 
denote the subgroup consisting of the isomorphic extension classes which admit a (not necessarily homomorphic) section, see \cite[Expos\'e XVII, \S A.2, \S A.3]{sga3-2}. The group $H^2_{x_{\mr{fl}}}(H_{p^r},G)$ can be identified with the group $\mr{Ext}_{\mr{s}}(H_{p^r}\times_Xx,G\times_Xx)$ by \cite[Expos\'e XVII, Prop. A.3.1]{sga3-2}. Hence it suffices to show that 
$$\mr{Ext}_{\mr{alg}}(H_{p^r}\times_Xx,G\times_Xx)=0.$$
By \cite[\S 9.2, Thm. 2]{b-l-r1}, $G\times_Xx$ is an extension of a unipotent group $U$ by a torus $T$ over $k$. Since a torus is smooth and $G\times_Xx$ is smooth over $k$, the unipotent quotient $U$ of $G\times_Xx$ is also smooth over $k$ by descent. By the exact sequence from \cite[\S A.2, Prop. A.2.1 a)]{sga3-2}, it suffices to show that
$$\mr{Ext}_{\mr{alg}}(H_{p^r}\times_Xx,T)=\mr{Ext}_{\mr{alg}}(H_{p^r}\times_Xx,U)=0.$$
By \cite[Expos\'e XVII, Prop. 7.1.1 ((b))]{sga3-2}, any extension of $H_{p^r}\times_Xx$ by $T$ must be of multiplicative type, therefore must be commutative. Since the base field $k$ is separable closed, such an extension must be trivial. Therefore the group $\mr{Ext}_{\mr{alg}}(H_{p^r}\times_Xx,T)$ is trivial. Since $H_{p^r}\times_Xx$ is connected and $U$ is smooth over $k$, we have $\mr{Ext}_{\mr{alg}}(H_{p^r}\times_Xx,U)=0$ by \cite[Expos\'e XVII, Thm. 5.1.1 (i) (d)]{sga3-2}. This finishes the proof of (2).

Part (3) can be proven in the same way as \cite[Lem. 3.20 (3)]{zha5}.
\end{proof}

Before making a corollary to Lemma \ref{2.3}, we make a lemma about the first direct image functor.

\begin{lem}\label{2.4}
Let $X$ be a locally  noetherian fs log scheme, and let $f:G_1\to G_2$ be an epimorphism of smooth commutative group schemes over the underlying scheme of $X$. Assume that $f$ is as in one of the following two cases.
\begin{enumerate}[(1)]
\item Both $G_1$ and $G_2$ are affine with geometrically connected fibers over the underlying scheme of $X$.
\item For $i=1,2$, $G_i$ is an extension of an abelian scheme $A_i$ by a torus $T_i$ over the underlying scheme of $X$, and $f$ is an isogeny. 
\end{enumerate}
Then $R^1\varepsilon_{\mr{fl}*}f:R^1\varepsilon_{\mr{fl}*}G_1\to R^1\varepsilon_{\mr{fl}*}G_2$ is surjective.
\end{lem}
\begin{proof}
Let $\mc{G}:=(\Gml/\Gm)_{X_{\mr{fl}}}$. By \cite[Thm. 3.14]{zha5}, we have 
$$R^1\varepsilon_{\mr{fl}*}G_i=\varinjlim_n\mc{H}om_X(\Z/n\Z(1),G_i)\otimes_\Z\mc{G}.$$
By \cite[Lem. A.1]{zha5}, the sheaf $\mc{H}om_X(\Z/n\Z(1),G_i)$ is representable by an \'etale quasi-finite group scheme over the underlying scheme of $X$ in both cases. Hence it suffices to check the surjectivity for the classical \'etale topology. Since the stalk at a point $x$ of $\mc{H}om_X(\Z/n\Z(1),G_i)\otimes_\Z\mc{G}$ agrees with $\mc{H}om_x(\Z/n\Z(1),G_i)\otimes_\Z\mc{G}$, we are further reduced to the case that the underlying scheme of $X$ is $\Spec k$ for a separable closed field $k$.

Now we deal with case (1). The group $G_i$ is an extension of a unipotent group $U_i$ by a torus $T_i$ by \cite[\S 9.2, Thm. 2]{b-l-r1}, and thus the map $f$ induces the following commutative diagram
$$\xymatrix{
0\ar[r] &T_1\ar[r]\ar[d]^{f'} &G_1\ar[r]\ar[d]^f &U_1\ar[r]\ar[d]^{f''} &0 \\
0\ar[r] &T_2\ar[r] &G_2\ar[r] &U_2\ar[r] &0
}$$
by \cite[Expos\'e XVII, Prop. 2.4 (i)]{sga3-2}.
Since $f$ is an epimorphism, so are $f'$ and $f''$ by \cite[Expos\'e XVII, Prop. 2.4 (ii)]{sga3-2}. Again by \cite[Expos\'e XVII, Prop. 2.4 (i)]{sga3-2}, the map 
\begin{equation}\label{eq2.1}
\varinjlim_n\mc{H}om_X(\Z/n\Z(1),G_1)\to\varinjlim_n\mc{H}om_X(\Z/n\Z(1),G_2)
\end{equation}
induced by $f$ is canonically identified with the map 
\begin{equation}\label{eq2.2}
\varinjlim_n\mc{H}om_X(\Z/n\Z(1),T_1)\to\varinjlim_n\mc{H}om_X(\Z/n\Z(1),T_2)
\end{equation}
induced by $f'$. Let $X_1$ (resp. $X_2$) be the character group of $T_1$ (resp. $T_2$), then the map (\ref{eq2.2}) can be further identified with
$$\mr{Hom}_{\Z}(X_1,\Z)\otimes_{\Z}\Q/\Z\to \mr{Hom}_{\Z}(X_2,\Z)\otimes_{\Z}\Q/\Z$$
which is surjective by the surjectivity of $f'$. Therefore $R^1\varepsilon_*f$ is surjective. 

At last we show case (2). If the characteristic of the separable closed field $k$ is zero, then the map (\ref{eq2.1}) can be identified with the map
$$G_1(k)_{\mr{tor}}\to G_2(k)_{\mr{tor}}$$
induced by $f$ on the torsion subgroups, which is clearly surjective. We are reduced to consider the case that the characteristic of $k$ is $p>0$. In this case the surjectivity of the coprime to $p$ part of the map (\ref{eq2.1}) can be proven in the same as in the characteristic zero case. Hence we are reduced to check that the surjectivity of the $p$-primary part
\begin{equation}\label{eq2.3}
\varinjlim_n\mc{H}om_X(\Z/p^n\Z(1),G_1)\to\varinjlim_n\mc{H}om_X(\Z/p^n\Z(1),G_2)
\end{equation}
of the map (\ref{eq2.1}). Let $G_1[p^\infty]_{\mr{m}}$ (resp. $G_2[p^\infty]_{\mr{m}}$) be the multiplicative part of the $p$-divisible group of $G_1$ (resp. $G_2$), then the map (\ref{eq2.3}) can be identified with the map
\begin{equation}\label{eq2.4}
\varinjlim_n\mc{H}om_X(\Z/p^n\Z(1),G_1[p^\infty]_{\mr{m}})\to\varinjlim_n\mc{H}om_X(\Z/p^n\Z(1),G_2[p^\infty]_{\mr{m}})
\end{equation}
induced by $f$. 

Now we digress to understand the $p$-divisible groups $G_1[p^\infty]_{\mr{m}}$ and $G_2[p^\infty]_{\mr{m}}$. Since $\mr{Hom}_X(T_1,A_2)=0$ by \cite[Thm. 4.3.2 (1)]{bri2}, the map $f$ induces the following commutative diagram
$$\xymatrix{
0\ar[r] &T_1\ar[r]\ar[d]^{f'} &G_1\ar[r]\ar[d]^f &A_1\ar[r]\ar[d]^{f''} &0 \\
0\ar[r] &T_2\ar[r] &G_2\ar[r] &A_2\ar[r] &0
}.$$
Since $f$ is an isogeny and there is no non-trivial homomorphism from abelian variety to torus, both $f'$ and $f''$ are isogeny. Let $r$ be the dimension of $T_1$ and let $s$ be the $p$-rank of the dual of $A_1$, then the dimension of $T_2$ is also $r$ and the $p$-rank of the dual of $A_2$ is also $s$. Since $k$ is separable closed, we have $G_i[p^\infty]_{\mr{m}}\cong (\Q_p/\Z_p(1))^{r+s}$ for $i=1,2$.

Now we can finish the proof. We have 
$$\varinjlim_n\mc{H}om_X(\Z/p^n\Z(1),G_1[p^\infty]_{\mr{m}})\cong (\Q_p/\Z_p)^{r+s}$$ for $i=1,2$. Since the map $G_1[p^\infty]_{\mr{m}}\to G_2[p^\infty]_{\mr{m}}$ induced by $f$ is an isogeny, the map (\ref{eq2.4}) is surjective. It follows that $R^1\varepsilon_{\mr{fl}*}f$ is surjective as well in case (2).
\end{proof}

\begin{cor}\label{2.5}
Let the notation and the assumptions be as in Theorem \ref{2.1} and Lemma \ref{2.3}. We further assume that $G$ is affine. Let $n$ be a  positive integer $n$ which is invertible on $X$, and let $G[n]$ denote the kernel of the multiplication by $n$ map on $G$. Then $G[n]$ is a quasi-finite \'etale group scheme over the underlying scheme of $X$ and
$$H^2_{\mr{kfl}}(X,G[n])\cong \Gamma(X,G[n](-2))\otimes_{\Z}P^{\mr{gp}}$$
\end{cor}
\begin{proof}
By the proof of \cite[Thm. 3.8 (3)]{zha5}, $G[n]$ is a quasi-finite \'etale group scheme over the underlying scheme of $X$, and we have a short exact sequence 
$$0\to G[n]\to G\xrightarrow{n}G\to0$$
of sheaves on both $(\mr{fs}/X)_{\mr{k\acute{e}t}}$ and $(\mr{fs}/X)_{\mr{kfl}}$. This exact sequence induces two exact sequences
$$0\rightarrow H_{\mr{kfl}}^1(X,G)\otimes_{\Z}\Z/n\Z\rightarrow H_{\mr{kfl}}^2(X,G[n])\rightarrow H_{\mr{kfl}}^2(X,G)[n]\rightarrow 0$$
and
$$0\rightarrow H_{\mr{k\acute{e}t}}^1(X,G)\otimes_{\Z}\Z/n\Z\rightarrow H_{\mr{k\acute{e}t}}^2(X,G[n])\rightarrow H_{\mr{k\acute{e}t}}^2(X,G)[n]\rightarrow 0.$$
Since $H_{\mr{fl}}^i(X,R^1\varepsilon_*G)\cong H_{\mr{\acute{e}t}}^i(X,R^1\varepsilon_*G)=0$ for any $i>0$, we get 
$$H_{\mr{kfl}}^1(X,G)\cong H_{\mr{fl}}^0(X,R^1\varepsilon_*G).$$
Thus the map $H_{\mr{kfl}}^1(X,G)\xrightarrow{n}H_{\mr{kfl}}^1(X,G)$ can be identified with 
$$H_{\mr{fl}}^0(X,R^1\varepsilon_*G)\xrightarrow{n}H_{\mr{fl}}^0(X,R^1\varepsilon_*G).$$
By Lemma \ref{2.4} (1), the map $R^1\varepsilon_*G\xrightarrow{n}R^1\varepsilon_*G$ is surjective. Since 
$$H_{\mr{fl}}^i(X,(R^1\varepsilon_*G)[n])\cong H_{\mr{\acute{e}t}}^i(X,(R^1\varepsilon_*G)[n])=0,$$
the map $H_{\mr{fl}}^0(X,R^1\varepsilon_*G)\xrightarrow{n}H_{\mr{fl}}^0(X,R^1\varepsilon_*G)$ is surjective. It follows that 
$$H_{\mr{kfl}}^1(X,G)\otimes_{\Z}\Z/n\Z=0,$$
and thus $H_{\mr{kfl}}^2(X,G[n])=H_{\mr{kfl}}^2(X,G)[n]$. By the proof of \cite[Thm. 3.8 (3)]{zha5}, the map $R^1\varepsilon_{\mr{\acute{e}t}*}G\xrightarrow{n}R^1\varepsilon_{\mr{\acute{e}t}*}G$ is also surjective. By a  similar argument as above, one can show that $H_{\mr{k\acute{e}t}}^1(X,G)\otimes_{\Z}\Z/n\Z=0$. Thus $H_{\mr{k\acute{e}t}}^2(X,G[n])=H_{\mr{k\acute{e}t}}^2(X,G)[n]$. Therefore
\begin{align*}
H_{\mr{kfl}}^2(X,G[n])=H_{\mr{kfl}}^2(X,G)[n]=H_{\mr{k\acute{e}t}}^2(X,G)[n]=&H_{\mr{k\acute{e}t}}^2(X,G[n])  \\
=&\Gamma(X,G[n](-2))\otimes_{\Z}\bigwedge^2P^{\mr{gp}},
\end{align*}
where the second equality follows from Lemma \ref{2.3} (3) and the last equality follows from \cite[Thm. 2.4]{k-n1}.
\end{proof}

\begin{thm}\label{2.6}
Let $X$ be a locally  noetherian fs log scheme, and let $G$ be a smooth commutative affine group scheme with geometrically connected fibers over the underlying scheme of $X$. 
\begin{enumerate}[(1)]
\item We have 
$$R^2\varepsilon_{\mr{fl}*}G=\varinjlim_{n} (R^2\varepsilon_{\mr{fl}*}G)[n]=\bigoplus_{\text{$l$ prime}}(R^2\varepsilon_{\mr{fl}*}G)[l^\infty],$$
where $(R^2\varepsilon_{\mr{fl}*}G)[n]$ denotes the $n$-torsion subsheaf of $R^2\varepsilon_{\mr{fl}*}G$ and $(R^2\varepsilon_{\mr{fl}*}G)[l^\infty]$ denotes the $l$-primary part of $R^2\varepsilon_{\mr{fl}*}G$ for a prime number $l$.
\item The $l$-primary part $(R^2\varepsilon_{\mr{fl}*}G)[l^\infty]$ is supported on the locus where $l$ is invertible.
\item If $n$ is invertible on $X$, then 
$$(R^2\varepsilon_{\mr{fl}*}G)[n]=R^2\varepsilon_{\mr{fl}*}G[n]=G[n](-2)\otimes_{\Z}\bigwedge^2(\Gml/\Gm)_{X_{\mr{fl}}}.$$
\end{enumerate}
\end{thm}
\begin{proof}
By Theorem \ref{1.1}, $R^2\varepsilon_{\mr{fl}*}G$ is torsion. Hence part (1) follows.

Part (2) follows from Lemma \ref{2.3} (3).

We are left with part (3).  By \cite[\S 3]{swa1}, in particular the part between \cite[Cor. 3.7]{swa1} and \cite[Thm. 3.8]{swa1}, we have cup-product for the higher direct image functors for the map of sites $\varepsilon_{\mr{fl}}:(\mr{fs}/X)_{\mr{kfl}}\rightarrow (\mr{fs}/X)_{\mr{fl}}$. The cup-product induces homomorphisms
$$G[n]\otimes_{\Z/n\Z}\bigwedge^2R^1\varepsilon_{\mr{fl}*}\Z/n\Z  
\rightarrow G[n]\otimes_{\Z/n\Z}R^2\varepsilon_{\mr{fl}*}\Z/n\Z\rightarrow R^2\varepsilon_{\mr{fl}*}G[n].$$
Since 
$$G[n]\otimes_{\Z/n\Z}\bigwedge^2R^1\varepsilon_{\mr{fl}*}\Z/n\Z=G[n](-2)\otimes_{\Z}\bigwedge^2(\Gml/\Gm)_{X_{\mr{fl}}},$$
we get a canonical homomorphism
$$G[n](-2)\otimes_{\Z}\bigwedge^2(\Gml/\Gm)_{X_{\mr{fl}}}\rightarrow R^2\varepsilon_{\mr{fl}*}G[n].$$
By Corollary \ref{2.5}, this homomorphism is an isomorphism. This finishes the proof of part (3).
\end{proof}

\begin{rmk}
In the case that $G$ is a torus, we have $(R^2\varepsilon_{\mr{fl}*}G)[n]=R^2\varepsilon_{\mr{fl}*}G[n]$ for any positive integer $n$ by \cite[Thm. 3.23 (2)]{zha5}. In Theorem \ref{2.6} we only have this identification for positive integer $n$ which is invertible on the base. One reason for this is that we always have a short exact sequence $0\to G[n]\to G\xrightarrow{n}G\to0$ for $n\geq1$ in the torus case.
\end{rmk}

\subsection{The case of certain finite flat group schemes}
\begin{thm}\label{2.7}
Let $X$ be a locally noetherian fs log scheme, and $F$ a finite flat group scheme over the underlying scheme of $X$. Then we have
\begin{enumerate}[(1)]
\item $R^2\varepsilon_{\mr{fl}*}F=\bigoplus_{\text{$l$ prime}}(R^2\varepsilon_{\mr{fl}*}F)[l^\infty]$, where $(R^2\varepsilon_{\mr{fl}*}F)[l^\infty]$ denotes the $l$-primary part of $R^2\varepsilon_{\mr{fl}*}F$.
\item The $l$-primary part $(R^2\varepsilon_{\mr{fl}*}F)[l^\infty]$ is supported on the locus where $l$ is invertible.
\item Assume $X$ is connected, then the order of $F$ is constant on $X$. We denote the order by $n$ and assume that $n$ is invertible on $X$. Then 
$$R^2\varepsilon_{\mr{fl}*}F=F(-2)\otimes_{\Z}\bigwedge^2(\Gml/\Gm)_{X_{\mr{fl}}},$$
where $F(-2)$ denotes the tensor product of $F$ with $\Z/n\Z(-2)$.
\end{enumerate}
\end{thm}
\begin{proof}
Part (1) is clear. 

Let 
$$0\to F\to G_1\xrightarrow{\theta} G_2\to 0$$
be the canonical smooth resolution of $F$ by smooth affine commutative group schemes, see \cite[Prop. 2.2.1]{beg1} or \cite[Thm. A.5]{mil2}. By the construction of loc. cit., the group scheme $G_1$ is actually the Weil restriction of $\Gm$ from the underlying scheme of $X$ to that of $F$. We claim that $G_1$ has geometrically connected fiber over $X$. It suffices to check this for the case that $X$ is a log point. Then the result follows from \cite[Prop. A.5.9]{c-g-p1}. It follows that $G_2$ also has connected fiber over $X$.

Applying the functor $\varepsilon_{\mr{fl}*}$, we get an exact sequence
\begin{equation}\label{eq2.5}
R^1\varepsilon_{\mr{fl}*}G_1\xrightarrow{R^1\varepsilon_{\mr{fl}*}\theta} R^1\varepsilon_{\mr{fl}*}G_2\to R^2\varepsilon_{\mr{fl}*}F\to R^2\varepsilon_{\mr{fl}*}G_1\xrightarrow{R^2\varepsilon_{\mr{fl}*}\theta} R^2\varepsilon_{\mr{fl}*}G_2.
\end{equation}
The map $R^1\varepsilon_{\mr{fl}*}\theta$ is surjective by Lemma \ref{2.4} (1). It follows that $R^2\varepsilon_{\mr{fl}*}F$ is a subsheaf of $R^2\varepsilon_{\mr{fl}*}G_1$. Since the $l$-primary part $(R^2\varepsilon_{\mr{fl}*}G_1)[l^\infty]$ of $R^2\varepsilon_{\mr{fl}*}G_1$ is supported on the locus where $l$ is invertible by Theorem \ref{2.6} (2), so is $(R^2\varepsilon_{\mr{fl}*}F)[l^\infty]$. This finishes the proof of part (2).

Now we prove part (3). Since $F$ is killed by its order $n$, so is the sheaf $R^2\varepsilon_{\mr{fl}*}F$. Let $(R^2\varepsilon_{\mr{fl}*}G_i)[n^\infty]$ denote the $n$-power torsion subsheaf of $R^2\varepsilon_{\mr{fl}*}G_i$ for $i=1,2$. Then the exact sequence (\ref{eq2.5}) induces an exact sequence
\begin{equation}\label{eq2.6}
0\to R^2\varepsilon_{\mr{fl}*}F\to (R^2\varepsilon_{\mr{fl}*}G_1)[n^\infty]\xrightarrow{R^2\varepsilon_{\mr{fl}*}\theta} (R^2\varepsilon_{\mr{fl}*}G_2)[n^\infty].
\end{equation}
We abbreviate $\bigwedge^2(\Gml/\Gm)_{X_{\mr{fl}}}$ as $\Lambda$, and let $G_i[n^r](-2)$ denote the tensor product of $G_i[n^r]$ with $\Z/n^r\Z(-2)$ for $i=1,2$. The exact sequence (\ref{eq2.6}) extends to a diagram
\begin{equation}\label{eq2.7}
\xymatrix{
0\ar[r] &F(-2)\otimes_\Z\Lambda\ar[r]\ar[d]^{\alpha_1} &\varinjlim_rG_1[n^r](-2)\otimes_\Z\Lambda\ar[r]\ar[d]^{\alpha_2}_\cong &\varinjlim_rG_2[n^r](-2)\otimes_\Z\Lambda\ar[d]^{\alpha_3}_\cong  \\
0\ar[r] &R^2\varepsilon_{\mr{fl}*}F\ar[r] &(R^2\varepsilon_{\mr{fl}*}G_1)[n^\infty]\ar[r]^{R^2\varepsilon_{\mr{fl}*}\theta} &(R^2\varepsilon_{\mr{fl}*}G_2)[n^\infty]
},
\end{equation}
where the vertical maps are given by the cup-product construction as in the proof of Theorem \ref{2.6} (3), the diagram is commutative by the functoriality of cup-product, and the maps $\alpha_2$ and $\alpha_3$ are isomorphisms by Theorem \ref{2.6} (3). In order to prove part (3), i.e. the map $\alpha_1$ is an isomorphism, it suffices to prove the lemma below by the five lemma.
\end{proof}

\begin{lem}
The first row of the diagram (\ref{eq2.7}) is exact.
\end{lem}
\begin{proof}
The multiplication by $n^r$ maps give rise to the following commutative diagram
$$\xymatrix{
0\ar[r] &F\ar[r]\ar[d]^{0} &G_1\ar[r]^\theta\ar[d]^{n^r} &G_2\ar[r]\ar[d]^{n^r} &0  \\
0\ar[r] &F\ar[r] &G_1\ar[r]^\theta &G_2\ar[r] &0
}$$
with exact rows, which induces an exact sequence
$$0\to F\to G_1[n^r]\to G_2[n^r].$$
We regard this sequence as an exact sequence of $\Z/n^r\Z$-modules and tensor it with $\Z/n^r\Z(-2)$, and thus get another exact sequence 
$$0\to F(-2)\to G_1[n^r](-2)\xrightarrow{\theta_r} G_2[n^r](-2).$$
Since the sheaf $\Lambda=\bigwedge^2(\Gml/\Gm)_{X_{\mr{fl}}}$ takes values in free abelian groups, tensoring with $\Lambda=\bigwedge^2(\Gml/\Gm)_{X_{\mr{fl}}}$ gives rise to an exact sequence
$$0\to F(-2)\otimes_\Z\Lambda\to G_1[n^r](-2)\otimes_\Z\Lambda\rightarrow G_2[n^r](-2)\otimes_\Z\Lambda.$$
Passing to direct limit, we get an exact sequence which is nothing but the first row of the diagram (\ref{eq2.7}).
\end{proof}

\subsection{The case of extensions of abelian schemes by tori}
\begin{thm}\label{2.9}
Let $X$ be a locally noetherian fs log scheme, and $G$ a group scheme over the underlying scheme of $X$ which is an extension of an abelian scheme $A$ by a torus $T$ over $X$. Then we have
\begin{enumerate}[(1)]
\item $R^2\varepsilon_{\mr{fl}*}G=\bigoplus_{\text{$l$ prime}}(R^2\varepsilon_{\mr{fl}*}G)[l^\infty]$, where $(R^2\varepsilon_{\mr{fl}*}G)[l^\infty]$ denotes the $l$-primary part of $R^2\varepsilon_{\mr{fl}*}G$.
\item We have $(R^2\varepsilon_{\mr{fl}*}G)[n]=R^2\varepsilon_{\mr{fl}*}G[n]$.
\item The $l$-primary part $(R^2\varepsilon_{\mr{fl}*}G)[l^\infty]$ is supported on the locus where $l$ is invertible.
\item If $n$ is invertible on $X$, then 
$$(R^2\varepsilon_{\mr{fl}*}G)[n]=R^2\varepsilon_{\mr{fl}*}G[n]=G[n](-2)\otimes_{\Z}\bigwedge^2(\Gml/\Gm)_{X_{\mr{fl}}},$$
where $G[n](-2)$ denotes the tensor product of $G[n]$ with $\Z/n\Z(-2)$.
\end{enumerate}
\end{thm}
\begin{proof}
By Theorem \ref{1.1}, $R^2\varepsilon_{\mr{fl}*}G$ is torsion. Hence part (1) follows.

Since $G$ is an extension of an abelian scheme by a torus, we have an exact sequence $0\to G[n]\to G\xrightarrow{n}G\to0$ of sheaves of abelian groups on $(\mr{fs}/X)_{\mr{kfl}}$ and $G[n]$ is a finite flat group scheme over the underlying scheme of $X$. Applying the functor $\varepsilon_{\mr{fl}*}$, we get an exact sequence
$$0\to R^1\varepsilon_{\mr{fl}*}G\otimes_{\Z}\Z/n\Z\to R^2\varepsilon_{\mr{fl}*}G[n]\to (R^2\varepsilon_{\mr{fl}*}G)[n]\to0.$$
By Lemma \ref{2.4} (2), we have $R^1\varepsilon_{\mr{fl}*}G\otimes_{\Z}\Z/n\Z=0$. Thus part (2) follows.

Since $G[n]$ is a finite flat group scheme, part (3) and part (4) follow from part (2) and Theorem \ref{2.7}.
\end{proof}

\subsection{Special cases of the base}
We can give more explicit description of the second higher direct image in some special cases of the base.

The first case concerns the ranks of the log structure at geometric points of the base. For an fs log scheme $X$ and a geometric point $x$ of $X$, by \textbf{the rank of the log structure of $X$ at $x$} we mean the rank of the free abelian group $(M_X^{\mr{gp}}/\mc{O}_X^\times)_x$.

\begin{thm}\label{2.10}
Let $X$ be a locally noetherian fs log scheme, and let $G$ be a commutative group scheme over the underlying scheme of $X$ which satisfies any of the following three conditions:
\begin{enumerate}[(i)]
\item $G$ is smooth, affine, and of geometrically connected fibers over the underlying scheme of $X$;
\item $G$ is finite, flat, and of order $n$ over the underlying scheme of $X$, and for any prime number $l$ the kernel of $G\xrightarrow{l^n}G$ is also finite flat over the underlying scheme of $X$;
\item $G$ is an extension of an abelian scheme by a torus over the underlying scheme of $X$.
\end{enumerate}
Let $Y\in (\mr{fs}/X)$ be such that the ranks of the stalks of the \'etale sheaf $M_Y^{\mr{gp}}/\mc{O}_Y^{\times}$ are at most one, and let $(\mr{st}/Y)$ be the full subcategory of $(\mr{fs}/X)$ consisting of strict fs log schemes over $Y$. Then the restriction of $R^2\varepsilon_{\mr{fl}*}G$ to $(\mr{st}/Y)$ is zero.
\end{thm}
\begin{proof}
By Theorem \ref{2.6} (1), Theorem \ref{2.7} (1), and Theorem \ref{2.9} (1), it suffices to show that the restriction of the $l$-primary part $(R^2\varepsilon_{\mr{fl}*}G)[l^\infty]$ to $(\mr{st}/Y)$ is zero for any prime number $l$. Since $(R^2\varepsilon_{\mr{fl}*}G)[l^\infty]$ is supported on the locus where $l$ is invertible by Theorem \ref{2.6} (2), Theorem \ref{2.7} (2), and Theorem \ref{2.9} (3), we are reduced to the case that $l$ is invertible on $Y$. By Theorem \ref{2.6} (3) and Theorem \ref{2.9} (4), we have
$$(R^2\varepsilon_{\mr{fl}*}G)[l^\infty]=\varinjlim_r G[l^r](-2)\otimes_{\Z}\bigwedge^2(\Gml/\Gm)_{X_{\mr{fl}}}$$
in cases (i) and (iii). 

In case (ii), we have a short exact sequence
$$0\to G[l^n]\to G\to G/G[l^n]\to0$$
of finite flat group schemes over the underlying scheme of $X$. Since the order of $G/G[l^n]$ is coprime to $l$, we get $(R^i\varepsilon_{\mr{fl}*}(G/G[l^n]))[l^\infty]=0$ for any $i>0$. Hence 
$$(R^2\varepsilon_{\mr{fl}*}G)[l^\infty]=(R^2\varepsilon_{\mr{fl}*}G[l^n])[l^\infty]=G[l^n](-2)\otimes_{\Z}\bigwedge^2(\Gml/\Gm)_{X_{\mr{fl}}}$$
by Theorem \ref{2.7} (3). Since the ranks of the stalks of the \'etale sheaf $M_Y^{\mr{gp}}/\mc{O}_Y^{\times}$ are at most one, the restriction of the sheaf $\bigwedge^2(\Gml/\Gm)_{X_{\mr{fl}}}$ to $(\mr{st}/Y)$ is zero. Therefore the restriction of the sheaf $(R^2\varepsilon_{\mr{fl}*}G)[l^\infty]$ to $(\mr{st}/Y)$ is zero in all the three cases. This finishes the proof.
\end{proof}

Next we investigate the cases involving the characteristic of the base.

\begin{thm}\label{2.11}
Let $X$ be a locally noetherian fs log scheme such that the underlying scheme of $X$ is a $\Q$-scheme, and let $G$ be as in any of the three cases of Theorem \ref{2.10}.
Then we have 
$$R^2\varepsilon_{\mr{fl}*}G=\varinjlim_{n}G[n](-2)\otimes_{\Z}\bigwedge^2(\Gml/\Gm)_{X_{\mr{fl}}}$$
in cases (i) and (iii), and
$$R^2\varepsilon_{\mr{fl}*}G=G(-2)\otimes_{\Z}\bigwedge^2(\Gml/\Gm)_{X_{\mr{fl}}}$$
in case (ii).
\end{thm}
\begin{proof}
Since the underlying scheme of $X$ is a $\Q$-scheme, the results follow from Theorem \ref{2.6}, Theorem \ref{2.7}, and Theorem \ref{2.9} easily.
\end{proof}

\begin{thm}\label{2.12}
Let $p$ be a prime number. Let $X$ be a locally noetherian fs log scheme such that the underlying scheme of $X$ is an $\F_p$-scheme, and $G$ a commutative group scheme over the underlying scheme of $X$ which satisfies any one of the following three conditions:
\begin{enumerate}[(i)]
\item $G$ is smooth, affine, and of geometrically connected fibers over the underlying scheme of $X$;
\item $G$ is finite, flat, and of order $n$ over the underlying scheme of $X$, and the kernel $G[p^n]$ of $G\xrightarrow{p^n}G$ is also finite flat over the underlying scheme of $X$;
\item $G$ is an extension of an abelian scheme by a torus over the underlying scheme of $X$.
\end{enumerate}
Then we have 
$$R^2\varepsilon_{\mr{fl}*}G=\varinjlim_{(r,p)=1}G[r](-2)\otimes_{\Z}\bigwedge^2(\Gml/\Gm)_{X_{\mr{fl}}}$$
in cases (i) and (iii), and 
$$R^2\varepsilon_{\mr{fl}*}G=(G/G[p^n])(-2)\otimes_{\Z}\bigwedge^2(\Gml/\Gm)_{X_{\mr{fl}}}$$
in case (ii).
\end{thm}
\begin{proof}
Since the underlying scheme of $X$ is an $\F_p$-scheme, $R^2\varepsilon_{\mr{fl}*}G$ is $p$-power torsion free by Theorem \ref{2.6} (2), Theorem \ref{2.7} (2), and Theorem \ref{2.9} (3). By Theorem \ref{1.1}, $R^2\varepsilon_{\mr{fl}*}G$ is torsion for case (i) and case (iii). Hence we have 
\begin{align*}
R^2\varepsilon_{\mr{fl}*}G=\varinjlim_{(r,p)=1}(R^2\varepsilon_{\mr{fl}*}G)[r]=&\varinjlim_{(r,p)=1}R^2\varepsilon_{\mr{fl}*}G[r]  \\
=&\varinjlim_{(r,p)=1}G[r](-2)\otimes_{\Z}\bigwedge^2(\Gml/\Gm)_{X_{\mr{fl}}}
\end{align*}
by Theorem \ref{2.6} (3) (resp. Theorem \ref{2.9} (4)) for case (i) (resp. (iii)).

For case (ii), we have a short exact sequence
$$0\to G[p^n]\to G\to G/G[p^n]\to0$$
of finite flat group schemes over the underlying scheme of $X$. This short exact sequence induces an exact sequence
$$R^2\varepsilon_{\mr{fl}*}G[p^n]\to R^2\varepsilon_{\mr{fl}*}G\to R^2\varepsilon_{\mr{fl}*}(G/G[p^n])\to R^3\varepsilon_{\mr{fl}*}G[p^n].$$
Since the order of $G/G[p^n]$ is coprime to $p$, hence invertible on $X$, we get 
$$R^i\varepsilon_{\mr{fl}*}(G/G[p^n])=(G/G[p^n])(-2)\otimes_{\Z}\bigwedge^2(\Gml/\Gm)_{X_{\mr{fl}}}$$
by Theorem \ref{2.7} (3), which is $p$-power torsion free. Since $G[p^n]$ is killed by $p^n$, $R^i\varepsilon_{\mr{fl}*}G[p^n]$ is $p^n$-torsion for any $i>0$. It follows that 
$$R^2\varepsilon_{\mr{fl}*}G\cong R^2\varepsilon_{\mr{fl}*}(G/G[p^n])=(G/G[p^n])(-2)\otimes_{\Z}\bigwedge^2(\Gml/\Gm)_{X_{\mr{fl}}}.$$
This finishes the proof for case (ii).
\end{proof}


\section{Examples}
\subsection{Discrete valuation rings}\label{subsec3.1}
Let $R$ be a Henselian discrete valuation ring with fraction field $K$ and residue field $k$. Let $\pi$ be a uniformizer of $R$, and we endow $X=\Spec R$ with the log structure associated to the homomorphism $\N\rightarrow R,1\mapsto\pi$. Let $x$ be the closed point of $X$ and $i$ the closed immersion $x\hookrightarrow X$, and we endow $x$ with the induced log structure from $X$. Let $\eta$ be the generic point of $X$ and $j$ the open immersion $\eta\hookrightarrow X$. 

Let $G$ be a commutative group scheme over $\Spec R$ which satisfies one of the following three conditions:
\begin{enumerate}[(i)]
\item $G$ is smooth and affine over $\Spec R$ and has geometrically connected fibers over $\Spec R$;
\item $G$ is finite, flat, and of order $n$ over $\Spec R$, and for any prime number $l$ the kernel of $G\xrightarrow{l^n}G$ is also finite flat over $\Spec R$;
\item $G$ is an extension of an abelian scheme by a torus over $\Spec R$.
\end{enumerate}
We consider the Leray spectral sequence
\begin{equation}\label{eq3.1}
H^s_{\mr{fl}}(X,R^t\varepsilon_{\mr{fl}*}G)\Rightarrow H^{s+t}_{\mr{kfl}}(X,G).
\end{equation}
We have 
\begin{align*}
R^1\varepsilon_{\mr{fl}*}G=&\varinjlim_n\mc{H}om_X(\Z/n\Z(1),G)\otimes_{\Z}(\Gml/\Gm)_{X_{\mr{fl}}}  
\end{align*}
by \cite[Thm. 3.14]{zha5}.
Then on $(\mr{st}/X)$, we have $(\Gml/\Gm)_{X_{\mr{fl}}}\cong i_*\Z$. Therefore
\begin{equation}\label{eq3.2}
\begin{split}
H^s_{\mr{fl}}(X,R^1\varepsilon_{\mr{fl}*}G)=&H^s_{\mr{fl}}(X,\varinjlim_n\mc{H}om_X(\Z/n\Z(1),G)\otimes_{\Z}i_*\Z) \\ 
=&H^s_{\mr{fl}}(x,\varinjlim_n\mc{H}om_x(\Z/n\Z(1),G\times_Xx))  \\
=&\varinjlim_n H^s_{\mr{fl}}(x,\mc{H}om_x(\Z/n\Z(1),G\times_Xx)) \\
=&\varinjlim_n H^s_{\mr{\acute{e}t}}(x,\mc{H}om_x(\Z/n\Z(1),G\times_Xx))
\end{split}
\end{equation}
for $s\geq0$, where the last equality follows from \cite[\href{https://stacks.math.columbia.edu/tag/0DDU}{Tag 0DDU}]{stacks-project} and the fact that $\mc{H}om_x(\Z/n\Z(1),G\times_Xx)$ is representable by an \'etale group scheme (see \cite[Lem. A.1]{zha5}). We also have 
\begin{equation}\label{eq3.3}
H^s_{\mr{fl}}(X,R^2\varepsilon_{\mr{fl}*}G)=H^s_{\mr{fl}}(X,0)=0
\end{equation}
for $s\geq0$ by Theorem \ref{2.10}. Then the spectral sequence (\ref{eq3.1}) gives rise to an exact sequence
\begin{equation}\label{eq3.4}
\begin{split}
0\to H^1_{\mr{fl}}(X,G)\to H^{1}_{\mr{kfl}}(X,G)\xrightarrow{\alpha} H^0_{\mr{fl}}(X,R^1\varepsilon_{\mr{fl}*}G)\to H^2_{\mr{fl}}(X,G)  \\
\to H^{2}_{\mr{kfl}}(X,G)\to H^1_{\mr{fl}}(X,R^1\varepsilon_{\mr{fl}*}G)\to H^3_{\mr{fl}}(X,G).
\end{split}
\end{equation}

\begin{prop}\label{3.1}
Assume that $k$ is a finite field, and $G$ satisfies condition (i). Let $T_x$ be the torus part of $G_x:=G\times_Xx$, and $N:=\mc{H}om_x(\Gm,T_x)$. Then we have 
$$H^1_{\mr{kfl}}(X,G)\cong \varinjlim_nH^0_{\mr{\acute{e}t}}(x,N\otimes_{\Z}\Z/n\Z)=H^0_{\mr{\acute{e}t}}(x,N\otimes_{\Z}\Q/\Z)$$
and
$$H^2_{\mr{kfl}}(X,G)\cong \varinjlim_nH^1_{\mr{\acute{e}t}}(x,N\otimes_{\Z}\Z/n\Z)=H^1_{\mr{\acute{e}t}}(x,N\otimes_{\Z}\Q/\Z).$$
\end{prop}
\begin{proof}
We have $H^s_{\mr{fl}}(X,G)=H^s_{\mr{\acute{e}t}}(X,G)=H^s_{\mr{\acute{e}t}}(x,G_x)$ by \cite[\href{https://stacks.math.columbia.edu/tag/0DDU}{Tag 0DDU}]{stacks-project} and \cite[Chap. III, Rmk. 3.11 (a)]{mil1}. Since $k$ as a finite field has cohomological dimension 1, we have $H^s_{\mr{\acute{e}t}}(x,G_x)=0$
for $s>1$. Since $G_x$ is connected and smooth, we also have $H^1_{\mr{\acute{e}t}}(x,G_x)=0$ by Lang's theorem. Then the exact sequence (\ref{eq3.4}) gives rise to $H^{s}_{\mr{kfl}}(X,G)\cong H^{s-1}_{\mr{fl}}(X,R^1\varepsilon_{\mr{fl}*}G)$ for $s=1,2$. By (\ref{eq3.2}), we get 
$$H^{s}_{\mr{kfl}}(X,G)\cong \varinjlim_n H^{s-1}_{\mr{\acute{e}t}}(x,\mc{H}om_x(\Z/n\Z(1),G_x))$$
for $s=1,2$. Since $G_x$ is smooth and connected, it fits into a short exact sequence $0\to T_x\to G_x\to U_x$ by \cite[\S 9.2, Thm. 2]{b-l-r1}, where $T_x$ is the torus part of $G_x$ and $U_x$ is unipotent. We have 
$$\mc{H}om_x(\Z/n\Z(1),G_x)=\mc{H}om_x(\Z/n\Z(1),T_x)$$
by \cite[Expos\'e XVII, Prop. 2.4 (i)]{sga3-2}. Applying the functor $\mc{H}om_x(-,T_x)$ to the short exact sequence 
$$0\to\Z/n\Z(1)\to \Gm\xrightarrow{n}\Gm\to0,$$
we get an exact sequence
$$0\to N\xrightarrow{n}N\to \mc{H}om_x(\Z/n\Z(1),T_x)\to \mc{E}xt^1_x(\Gm,T_x).$$
We have $\mc{E}xt^1_x(\Gm,T_x)=0$ by \cite[Expos\'e VIII, Prop. 3.3.1]{sga7-1}. Therefore $\mc{H}om_x(\Z/n\Z(1),T_x)=N\otimes_{\Z}\Z/n\Z$. It follows that
$$H^{s}_{\mr{kfl}}(X,G)\cong \varinjlim_n H^{s-1}_{\mr{\acute{e}t}}(x,N\otimes_{\Z}\Z/n\Z)=H^{s-1}_{\mr{\acute{e}t}}(x,N\otimes_{\Z}\Q/\Z)$$
for $s=1,2$.
\end{proof}

\begin{prop}\label{3.2}
Assume that $k$ is a finite field, and $G$ satisfies condition (ii). Let $G_x^{\mr{mul}}$ be the maximal multiplicative subgroup of $G_x:=G\times_Xx$ whose existence is guaranteed by \cite[Expos\'e XVII, Thm. 7.2.1]{sga3-2}. Let $n$ be an integer which kills $G$, and let $N:=\mc{H}om_x(\Z/n\Z(1),G_x^{\mr{mul}})$. Then we have 
$$H^1_{\mr{kfl}}(X,G)\cong H^1_{\mr{fl}}(X,G)\oplus H^0_{\mr{\acute{e}t}}(x,N)$$
and
$$H^2_{\mr{kfl}}(X,G)\cong H^1_{\mr{\acute{e}t}}(x,N).$$
\end{prop}
\begin{proof}
By \cite[Chap. III, Lem. 1.1]{mil2}, we have $H^s_{\mr{fl}}(X,G)=0$ for $s\geq2$. Then the exact sequence (\ref{eq3.4}) and the identification (\ref{eq3.2}) together give rise to a short exact sequence
$$0\to H^1_{\mr{fl}}(X,G)\to H^{1}_{\mr{kfl}}(X,G)\to \varinjlim_r H^0_{\mr{\acute{e}t}}(x,\mc{H}om_x(\Z/r\Z(1),G_x)) \to0$$
and an isomorphism
$$H^{2}_{\mr{kfl}}(X,G)\cong \varinjlim_rH^1_{\mr{\acute{e}t}}(x,\mc{H}om_x(\Z/r\Z(1),G_x)).$$
Since $n$ kills $G$, we get 
$$\varinjlim_rH^s_{\mr{\acute{e}t}}(x,\mc{H}om_x(\Z/r\Z(1),G_x))=H^s_{\mr{\acute{e}t}}(x,\mc{H}om_x(\Z/n\Z(1),G_x))$$
for $s=0,1$. Since $G_x/G_x^{\mr{mul}}$ is unipotent, we get 
$$\mc{H}om_x(\Z/n\Z(1),G_x)=\mc{H}om_x(\Z/n\Z(1),G_x^{\mr{mul}})=N$$
by \cite[Expos\'e XVII, Prop. 2.4]{sga3-2}. Thus we get a short exact sequence
\begin{equation}\label{eq3.5}
0\to H^1_{\mr{fl}}(X,G)\to H^1_{\mr{kfl}}(X,G)\to H^0_{\mr{\acute{e}t}}(x,N)\to0
\end{equation}
and an isomorphism
$$H^2_{\mr{kfl}}(X,G)\cong H^1_{\mr{\acute{e}t}}(x,N).$$
The short exact sequence (\ref{eq3.5}) is actually split by \cite[App. D, Lem. D.1]{w-z1}. Thus we get 
$$H^1_{\mr{kfl}}(X,G)\cong H^1_{\mr{fl}}(X,G)\oplus H^0_{\mr{\acute{e}t}}(x,N).$$
\end{proof}

\begin{rmk}
Let the situation be as in Proposition \ref{3.2}. The restriction map 
$H^1_{\mr{fl}}(X,G)\to H^1_{\mr{fl}}(\Spec K,G\times_X\Spec K)$ is injective by \cite[Chap. III, Lem. 1.1]{mil2}. The analogue for the Kummer log flat topology is also true, i.e. the restriction map $H^1_{\mr{kfl}}(X,G)\to H^1_{\mr{fl}}(\Spec K,G\times_X\Spec K)$ is injective, see \cite[Prop. 3.6]{gil1}.
\end{rmk}

\begin{prop}\label{3.3}
Assume that $k$ is a finite field, and $G$ satisfies condition (iii). Then we have 
$$H^s_{\mr{kfl}}(X,G)\cong \varinjlim_nH^{s-1}_{\mr{\acute{e}t}}(x,\mc{H}om_x(\Z/n\Z(1),G_x))$$
for $s=1,2$.
\end{prop}
\begin{proof}
The same argument in the beginning of the proof of Proposition \ref{3.1} shows also $$H^s_{\mr{fl}}(X,G)=H^s_{\mr{\acute{e}t}}(X,G)=H^s_{\mr{\acute{e}t}}(x,G_x)=0$$
for $s\geq1$. Then the result follows from the exact sequence (\ref{eq3.4}) and the identification (\ref{eq3.2}).
\end{proof}

\begin{rmk}
By Proposition \ref{3.1}, Proposition \ref{3.2}, and Proposition \ref{3.3}, we reduce the computation of $H^s_{\mr{kfl}}(X,G)$ for $s=1,2$ to the computation of certain \'etale cohomology groups over the finite field $k$, which can be easily computed via Galois cohomology over $k$.
\end{rmk}

\subsection{Global Dedekind domains}
Through this subsection, let $K$ be a global field. When $K$ is a number field, $X$ denotes the spectrum of the ring of integers in $K$, and when $K$ is a function field, $k$ denotes the field of constants of $K$ and $X$ denotes the unique connected smooth projective curve over $k$ having $K$ as its function field. Let $S$ be a finite set of closed points of $X$, $U:=X-S$, $j:U\hookrightarrow X$, and $i_x:x\hookrightarrow X$ for each closed point $x\in X$. We endow $X$ with the log structure $j_{*}\mc{O}^{\times}_U\cap\mc{O}_X\rightarrow \mc{O}_X$. In the case of number field, let $S_{\infty}:=S\cup\{\text{infinite places of $K$}\}$, and in the case of function field, we just let $S_{\infty}:=S$.

\begin{prop}
Let $X$ be as above. Let $G$ be a commutative group scheme over the underlying scheme of $X$ which satisfies any of the following three conditions:
\begin{enumerate}[(i)]
\item $G$ is smooth and affine over the underlying scheme of $X$, and has geometrically connected fibers;
\item $G$ is finite flat of order $n$ over the underlying scheme of $X$, and for any prime number $l$ the kernel of $G\xrightarrow{l^n}G$ is also finite flat over the underlying scheme of $X$;
\item $G$ is an extension of an abelian scheme by a torus over the underlying scheme of $X$.
\end{enumerate}
For $x\in S$, let $G_x:=G\times_Xx$. Then we have the following exact sequence
\begin{equation}\label{eq3.6}
\begin{split}
0\rightarrow &H^1_{\mr{fl}}(X,G)\rightarrow H^1_{\mr{kfl}}(X,G)\xrightarrow{\alpha} H^0_{\mr{fl}}(X,R^1\varepsilon_{\mr{fl}*}G)  \\
\rightarrow &H^2_{\mr{fl}}(X,G)\rightarrow H^2_{\mr{kfl}}(X,G)\rightarrow H^1_{\mr{fl}}(X,R^1\varepsilon_{\mr{fl}*}G)  \\
\rightarrow &H^3_{\mr{fl}}(X,G)\rightarrow H^3_{\mr{kfl}}(X,G),
\end{split}
\end{equation}
and 
$$H^i_{\mr{fl}}(X,R^1\varepsilon_{\mr{fl}*}G)\cong \bigoplus_{x\in S} H^i_{\mr{\acute{e}t}}(x,\varinjlim_n\mc{H}om_x(\Z/n\Z(1),G_x))$$
for $i=0,1$.
\end{prop}
\begin{proof}
On $(\mr{st}/X)$, we have $R^2\varepsilon_{\mr{fl}*}G=0$
by Theorem \ref{2.10}. Thus the Leray spectral sequence
$$H^s_{\mr{fl}}(X,R^t\varepsilon_{\mr{fl}*}G)\Rightarrow H^{s+t}_{\mr{kfl}}(X,G)$$ gives rise to the long exact sequence (\ref{eq3.6}). 

We have 
$$R^1\varepsilon_{\mr{fl}*}G=\bigoplus_{x\in S}i_{x,*}\varinjlim_n\mc{H}om_x(\Z/n\Z(1),G_x)$$
by \cite[Thm. 3.14]{zha5}. 
It follows that 
\begin{align*}
H^i_{\mr{fl}}(X,R^1\varepsilon_{\mr{fl}*}G)=&H^i_{\mr{fl}}(X,\bigoplus_{x\in S}i_{x,*}\varinjlim_n\mc{H}om_x(\Z/n\Z(1),G_x))   \\
=&\bigoplus_{x\in S} H^i_{\mr{fl}}(x,\varinjlim_n\mc{H}om_x(\Z/n\Z(1),G_x)) \\
\cong &\bigoplus_{x\in S} H^i_{\mr{\acute{e}t}}(x,\varinjlim_n\mc{H}om_x(\Z/n\Z(1),G_x)),
\end{align*}
where the last isomorphism follows from \cite[Lem. A.1]{zha5}.
\end{proof}

In the case that $G=\Gm$, the map $\alpha$ of the exact sequence (\ref{eq3.6}) has been shown to be surjective in \cite[\S 5.2]{zha5}. The map $\alpha$ of the exact sequence (\ref{eq3.4}) is also surjective if the residue field $k$ of the base ring $R$ is finite by Proposition \ref{3.1}, Proposition \ref{3.2}, and Proposition \ref{3.3}. It is not clear to the author at this moment if the map $\alpha$ of (\ref{eq3.6}) is surjective beyond the case $G=\Gm$. The method of \cite[\S 5.2]{zha5} does not work even for general tori.

\section*{Acknowledgement}
The author thanks Professor Ulrich G\"ortz for very helpful discussions. 
\bibliographystyle{alpha}
\bibliography{bib}

\end{document}